\documentclass[11pt]{article}

\usepackage{latexsym}
\usepackage{amsthm}
\usepackage{amsfonts}
\usepackage{amssymb}
\usepackage{amsmath}
\usepackage{graphicx}
\usepackage{xcolor}
\usepackage{amssymb}
\usepackage{multirow}
\usepackage{titletoc}
\usepackage{hyperref}
\usepackage{subfigure}
\usepackage{booktabs}
\usepackage{verbatim}
\usepackage{url}
\usepackage[affil-it]{authblk}

\usepackage[top=1in, bottom=1.5in, left=1in, right=1in]{geometry}
\usepackage[algo2e,ruled,linesnumbered,vlined]{algorithm2e}

\usepackage[style=authoryear,backend=bibtex,firstinits=true,maxnames=3,url=false,doi=false,isbn=false]{biblatex}
\bibliography{fault}

\newtheorem{theorem}{Theorem}%[section]

\newtheorem{corollary}[theorem]{Corollary}
\newtheorem{example}{Example}%[section]
\newtheorem{remark}{Remark} % Need amsthm

  {\begin{itemize}%
    \setlength{\topsep}{0pt}%
    \setlength{\itemsep}{0pt}%
    \setlength{\parskip}{0pt}}%
  {\end{itemize}}

\begin{document}

\title{An Error-Resilient Redundant Subspace Correction Method}

\author{Tao Cui%
  \thanks{Email: \texttt{tcui@lsec.cc.ac.cn}}}
\affil{LSEC, Academy of Mathematics and System Sciences, Beijing, China}

\author{Jinchao Xu%
  \thanks{Email: \texttt{xu@math.psu.edu}}}
\affil{Department of Mathematics, Pennsylvania State University, PA, USA}

\author{Chen-Song Zhang%
  \thanks{Email: \texttt{zhangcs@lsec.cc.ac.cn}}}
\affil{NCMIS \& LSEC, Academy of Mathematics and System Sciences, Beijing, China}

\date{\today}

\maketitle

\begin{abstract}
As we stride toward the exascale era, due to increasing complexity of supercomputers, hard and soft errors are causing more and more problems in high-performance scientific and engineering computation. In order to improve reliability (increase the mean time to failure) of computing systems, a lot of efforts have been devoted to developing techniques to forecast, prevent, and recover from errors at different levels, including architecture, application, and algorithm. In this paper, we focus on algorithmic error resilient iterative linear solvers and introduce a redundant subspace correction method. Using a general framework of redundant subspace corrections, we construct iterative methods, which have the following properties: (1)~Maintain convergence when error occurs assuming it is detectable; (2)~Introduce low computational overhead when no error occurs; (3)~Require only small amount of local (point-to-point) communication compared to traditional methods and maintain good load balance; (4)~Improve the mean time to failure. With the proposed method, we can improve reliability of many scientific and engineering applications. Preliminary numerical experiments demonstrate the efficiency and effectiveness of the new subspace correction method.
\end{abstract}


%%%%%%%%%%%%%%%%%%%%%%%%
%
\newpage
\setcounter{tocdepth}{1}
\tableofcontents
\bigskip \bigskip
%
%!TEX root = main.tex

\section{Introduction}\label{sec:intro}

Simulation-based scientific discovery and engineering design demand extreme computing power and high-efficiency algorithms. This demand is one of the main driving forces to pursuit of extreme-scale computer hardware and software during the last few decades~\parencite{Keyes2011}. Large-scale HPC installations are interrupted by data corruptions and hardware failures with increasing frequency~\parencite{Miskov-zivanov2007} and it becomes more and more difficult to maintain a reliable computing environment. It has been reported that the ASCI~Q computer (12,288 EV-68 processors) in the Los Alamos National Laboratory experienced 26.1 radiation-induced CPU failures per week~\parencite{Michalak2005} and a BlueGene/L (128K processors) experiences one soft error in its L1 cache every 4--6 hours due to radioactive decay in lead solder~\parencite{Bronevetsky2008a}. 

Computer dependability is, in short, a property that reliable results can be justifiably achieved; see, for example, \cite{Laprie1995}. Without promising reliability of a computer system, no application can promise anything about the final outcome. Design computing systems that meet high reliability standards, without exceeding fixed power budgets and cost constraints, is one of the fundamental challenges that present and future system architects face. It has become increasingly important for algorithms to be well-suited to the emerging parallel hardware architectures. Co-design of architecture, application, and algorithm is particularly important given that researchers are trying to achieve exascale ($10^{18}$ floating-point operations per second) computing~\parencite{Mukherjee2005,Abts2006,Dongarra2011}. To ensure robust and resilient execution, future systems will require designers across all layers (hardware, software, and algorithm) of the system stack to integrate design techniques adaptively~\parencite{Reddi2012}. 

As we enter the multi-petaflop era, frequency of a single CPU core does not increase beyond certain critical value. On the other hand, the number of computing cores in supercomputers is growing exponentially, which results in higher and higher system complexity. For example, in the recent released HPC Top 500 list~(\url{Top500.org}), the Tianhe-2 system at the National Supercomputing Center in Guangzhou has claimed the first spot in the Top 500. Tianhe-2 consists of 16,000 computer nodes, each comprising two Intel Ivy Bridge Xeon processors and three Xeon Phi coprocessors (3,120,000 processing cores and 1.37TB RAM in total). Tianhe-2 delivers 33.86 petaflops of sustained performance on the HPL benchmark, which is about 61\% of its theoretical peak performance. 

All components of a computing system (hardware and software) are subject to errors and failures. Inevitably, more complex the system, lower the reliability. Exascale computing systems are expected to be consist of massive number of computing nodes, processing cores, memory chips, disks, and network switches. It is projected that the Mean Time To Failure (MTTF) for some components of an exascale system will be in the minutes range. Fail-stop process failures is noticeable and is a common type of hardware failures on large computing systems, where the failed process stops working or responding and it will cause all data associated with the failed process lost. Soft errors (bit flips) caused by cosmic radiation and voltage fluctuation are another type of significant threads to long-running distributed applications. Large cache structures in modern multicore processors are particularly vulnerable to soft errors. Recent studies~\parencite{Bronevetsky2008a,Shantharam2011,Malkowski2010} show that soft errors could have very different impact on applications, from no effect at all or silent error to application crashes. 

For many PDE-based applications, solution of linear systems often takes most of the computing time (usually more than $80\%$ of wall-time for large simulations). Providing low overhead and scalable fault-tolerant linear solvers (preconditioners) is the key to improve reliability of these applications. Fault-tolerant iterative methods have been considered and analyzed by many researchers; see \cite{roy1993fault,Hoemmen,Shantharam2012} and references therein. Other fault-tolerant techniques in the field of numerical linear algebra can also be applied to iterative solvers~\parencite{Chen2008}. Most of existing fault-tolerant techniques fall into the following three categories:

{\it 1. Hardware-Based Fault Tolerance.}
Memory errors are one of the most common reasons of hardware crashes; see~\cite{Mukherjee2005,zhang2005computing} and references therein. Impact of soft errors in caches on the resilience and energy efficiency of sparse iterative methods are analyzed in~\cite{Bronevetsky2008a}. Hardware-based error detection and correction has been employed on different levels to improve system reliability. Different kinds of Error Correcting Code (ECC) schemes have been employed to protect the memory data from single or multiple bit flips. However, using more complex ECC schemes not only result in higher cost in hardware and energy, but also undermine the performance~\parencite{Malkowski2010}. 

{\it 2. Software-Based Fault Tolerance.}
The most important form of software fault tolerance techniques is probably checkpointing; see~\cite{Treaster2005} and references therein for details. If a failure occurs in one of the independent components, the directly affected parts of the system or the whole system is restarted and rolled back to a previously-stored safe state. The checkpointing and restarting techniques ensure that the internal state of recovered process conforms to the state before failure. There are several ways to design checkpoints, such as disk checkpointing, diskless checkpointing, and message logging~\parencite{plank1998diskless,langou2007recovery,Nassar2008}. Checkpoint/restart is usually applied to treat fail-stop failures because it is able to tolerate the failure of the whole system. However, the overhead associated with this approach is also very high---If a single process fails, the whole application needs to be restarted from the last stored state. Another approach is to utilize optimizing compilers to improve resilience; see, for example, \cite{chen2005compiler,li2005improving}.

{\it 3. Algorithm-Based Fault Tolerance.}
Algorithm-based fault tolerance (ABFT) schemes based on various implementations of checksum are proposed originally by \cite{Huang1984}. Later this idea was extended to detect and correct errors for matrix operations such as addition, multiplication, scalar product, LU-decomposition, and transposition; see, for example, \cite{Luk1986,boley1992algorithmic}. Another interesting work worth-noticing is an algorithm-based fault tolerant technique for fail-stop type of failures and its applications in ScaLAPACK~\parencite{Chen2008}. Error resilient direct solvers have recently been considered when single and multiple silent errors are occurred in \cite{Du2011} and in \cite{Du2012}, respectively. Fault-tolerant iterative methods such as SOR, GMRES, and CG for sparse linear systems have also been considered in~\cite{roy1993fault,Hoemmen,Shantharam2012} (in the event when there is at most one error). Selective reliability for iterative methods can be achieved using the ideas by \cite{Hoemmen}. \cite{stoyanov2013numerical} propose a new analytic approach for improving resilience of iterative methods with respect to silent errors by rejecting large hardware error propagation.

In this paper, we focus on resilient iterative solvers/preconditioners from a completely different perspective. Our main goal is to increase mean time to failure (MTTF) in the algorithm level by introducing local redundancy to the iterative procedure. We first introduce a virtual machine model, based on which we propose a framework of space decomposition and subspace correction method to design iterative methods that are reliable in response to errors. The general idea of subspace correction is to use a divide and conquer strategy to decompose the original solution space into the summation of a number of subspaces and then to make corrections on subspaces in an appropriate fashion. We mainly explore the intrinsic fault/error tolerance features of the method of subspace corrections:
\begin{itemize}
\item In the implementation of subspace correction method, we introduce redundant subspaces locally and make an appropriate mapping between subspaces and processors;
\item The proposed iterative algorithm still converges when single or multiple processes fail and it does not introduce heavy overhead in case no error occurs;
\item The proposed algorithm can be combined with existing hardware, software, and algorithm based fault tolerant techniques to improve reliability of spare-solver related applications. 
\end{itemize}

The rest of the paper is organized as follows: In Section~\ref{sec:virtual}, we describe a virtual machine model which will be used in the numerical experiments. In Section~\ref{sec:psc}, we discuss a parallel subspace correction method framework. In Section~\ref{sec:msc}, we discuss a multiplicative subspace correction method. In Section~\ref{sec:numerics}, we give some preliminary numerical results to test the proposed algorithms. And we conclude the paper with a few general remarks in Section~\ref{sec:conclusion}.
%
%!TEX root = main.tex

\section{A virtual machine model}\label{sec:virtual}

In order to describe our algorithm framework, we need to introduce a simplified reliability model based on the seven-level model proposed by Parhami~\parencite{Parhami1994,Parhami1997}. In our model, we assume that an application could be in one of the four states---ideal, faulty, erroneous, or failed; see Figure~\ref{fig:state}. %
\begin{figure}[h!!] %  figure placement: here, top, bottom, or page
   \centering
   \includegraphics[width=0.2\linewidth]{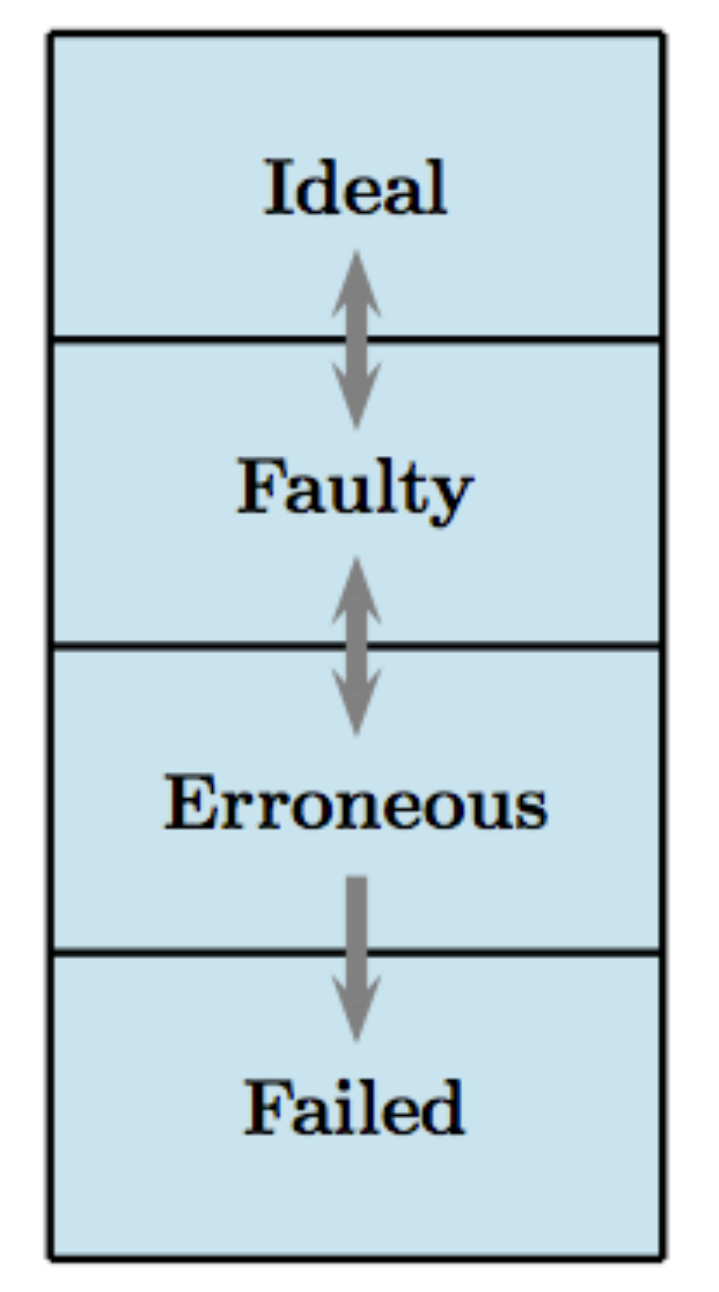} 
   \caption{System states in a simplified reliability model}
   \label{fig:state}
\end{figure}

Models of reliability have been also discussed by \cite{Hoemmen}. Notice that in our model, we distinguish fault and error. These terms are not exactly the same as the ones other people might be using where fault and error are usually interchangeable. We now describe these four states in details:
\begin{itemize}
\item Ideal state is the reliable operating condition under which expected output can be justifiably obtained.
\item A fault refers to an abnormal operating condition of the computer system due to a defective hardware or software. A fault could be transient or permanent---A transient fault is some incorrect data which affects the application temporarily and will be replaced by correct data in later time (e.g., a bit flip in cache which will be flushed later by the data in main memory). On the other hand, a permanent fault stands for incorrect data which will not be changed automatically (e.g., incorrect data in the main memory). A fault may not eventually cause error(s) (e.g., a bit-flip in cache might never be used); only if a fault is actually exercised, it may contaminate the data flow and cause errors. 
\item An error could be ``hard'' or ``soft'': A hard error is due to hardware failures (or unusual delays) and may be caused by a variety of phenomena, which include, but are not limited to, an unresponsive network switch or an operating system crashing; A soft error, on the other hand, is an one-time event, such as a bit-flip in main memory (and this bit is actually used in the application) and a logic circuit output error, that corrupts a computing system’s state but not its overall functionality. This concept of ``error'' can also be extended for the case when a node does not respond within an expected time period. Errors can be detected and corrected by the application in our model. 
\item A failed state means that some part of or whole application does not produce the expected results. As long as a system enters the ``failed'' state, interference from outside is necessary to fix the problem and the program itself cannot do anything to fix it. Resilience is a measure of the ability of a computing system and its applications to continue working in the presence of fault and error. 
\end{itemize}

Based on the reliability model described above, we introduce a virtual machine (VM), that ensures isolation of possibly unreliable phases of execution. A virtual machine can support individual processes or a complete system depending on the abstraction level where virtualization occurs~\parencite{Smith2005}. The concept of virtualization can be applied in various places, for example subsystems such as disks or an entire cluster. To implement a virtual machine, developers add a software layer to a real machine to support the desired architecture. By doing so, a VM can circumvent real machine compatibility and hardware resource constraints.

Due to defective hardwares and/or faulty data, a computer system could be compromised by errors. In a distributed memory cluster system, there could be deadlocks and other failures due to unresponsive computer nodes. In this conceptive VM under consideration, an error can be detected and resolved by system- or user-level error correction mechanisms. For example, a hanging guest process can be killed and resubmitted\footnote{A static Message Passing Interface (MPI) program has very limited job control and a single failed processor could cause the whole application to fail. Hence, the assumption {\bf A1} might not be satisfied for the current MPI standard. However, in the dynamic MPI standard, this could be implemented in practice~\parencite{Fagg2000}. Fault-tolerant MPI has been discussed by~\cite{Gropp2004}.
}; a bit-flip data error in the memory can be corrected by ECC. 

For proof-of-concept, we assume that our virtual machine guarantees the following reliability properties: 
\begin{itemize}
\item[\bf A1.] At any specific time in $(0,T]$ during the computation, there could be at most one processing unit in the erroneous/failed phase. Note that this assumption can be relaxed later on in \S\ref{sec:improve}.
\item[\bf A2.] An erroneous processing unit $U_i$ can be detected and corrected within a fixed amount of time.
\item[\bf A3.] A processing unit could be in any state for arbitrarily long time. For example, it could take more time to fix an erroneous or failed process than the actual computing time of the application. 
\end{itemize}
Depending on the programming model, a processing (or computing) unit could be a processing core, a multicore processor, or a computing node of a cluster. 

%
%!TEX root = main.tex

\section{Method of subspace corrections}\label{sec:psc}

Let $(\cdot, \cdot)$ be the $L^2$-inner product on $\Omega \subset \mathbb{R}^d$ ($d=1,2,3$) and a $n$-dimensional vector space $V$; its induced norm is denoted by $\| \cdot \|$. Let $A$ be a symmetric positive definite (SPD) operator on $V$, i.e., $A^T=A$ and $(Av,v) > 0$ for all $v \in V\backslash\{0\}$. The adjoint of $A$ with respect to $(\cdot, \cdot)$, denoted by $A^T$, is defined by $(Au,v) = (u,A^Tv)$ for all $u, v\in V$. As $A$ is SPD with respect to $(\cdot,\cdot)$, the bilinear form $(A\cdot,\cdot)$ defines an inner product on $V$, denoted by $(\cdot,\cdot)_{A}$, and the induced norm of $A$ is denoted by $\| \cdot \|_{A}$. The adjoint of $A$ with respect to $(\cdot,\cdot)_A$ is denoted by $A^*$. In this paper, we consider solution methods for the linear equation
\begin{equation}\label{eqn:linear_system}
Au=f.
\end{equation}

\subsection{Spatial Partition}

Suppose the computational domain $\Omega$ has been one-dimensionally\footnote{This assumption is only for the sake of discussion and can be removed easily.} partitioned into several subdomains $D_1, \ldots, D_N$ and each of these subdomains is owned by one processing (or computing) unit; see Figure~\ref{fig:partition0} (Left). 
Note that, although we use geometric partitioning to demonstrate the ideas, the method is  applicable to the algebraic versions. These simplifications (including the geometric domain decomposition assumption) have been made to make the discussion easier and are not essential. 

In general, we can view this partition in an algebraic setting: Let $D$ be the set of all indices for the degrees of freedom (DOFs) (number of the DOFs is assumed to be $n$) and 
$$
D := \{1,2,\ldots,n\} = \bigcup_{i=1}^N D_i.
$$
be a partition of $D$ into $N$ disjoint, nonempty subsets. For each $D_i$ we consider a nested sequence of larger sets $D_i^\delta$ with 
$$
D_i = D_i^0 \subseteq D_i^1 \subseteq D_i^2 \subseteq \cdots \subseteq D,
$$
where the nonnegative integer $\delta$ is the level of overlaps. 

Suppose the vector space $V$ be the solution space on $D$. And, $V$ is provided with a space decomposition
\begin{equation}\label{eqn:decomp}
V = \sum_{i=1}^N V_i,
\end{equation}
where the nonempty subspaces $V_i \subseteq V$ associated to the unknowns in the set $D_i^\delta$. To solve for the degrees of freedom in $D_i$, we might need data in $D^\delta_i$. We assume that all the necessary data for $D_i^\delta$ is owned by the processing unit $U_i$ for each $i$. With abuse of notation, we call this set of data $D_i^\delta$ as well.
\begin{figure}[h] %  figure placement: here, top, bottom, or page
   \centering
   \includegraphics[width=0.85\linewidth]{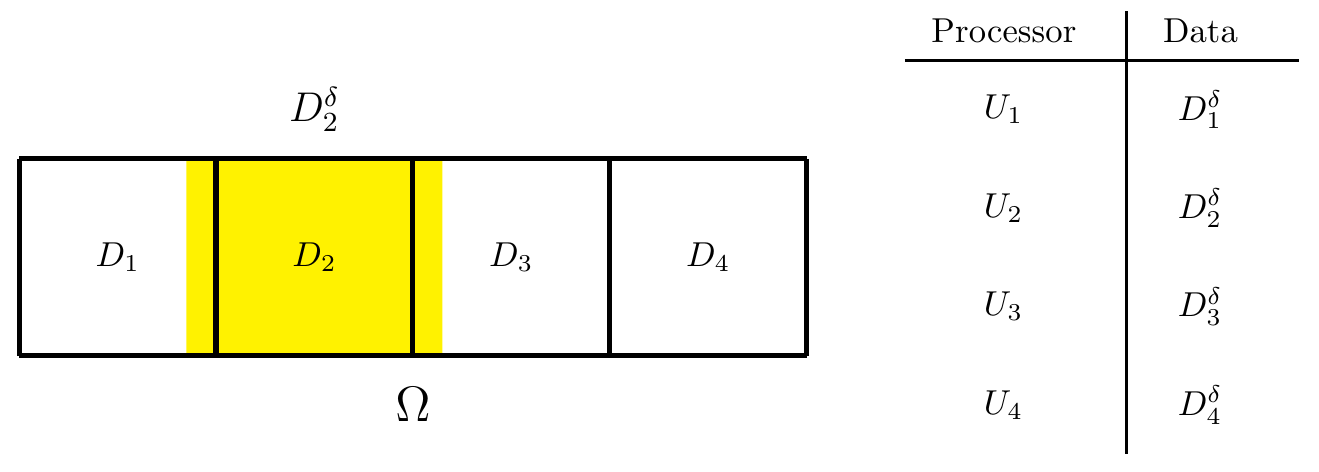} 
   \caption{Partition of the physical domain for overlapping additive Schwarz methods}
   \label{fig:partition0}
\end{figure}

\subsection{Subspace Corrections}

To solve large-scale linear systems arising from partial differential equations (PDEs), preconditioned iterative methods are usually employed~\parencite{Hackbusch.Hackbusch.1994fk}. It is well-known that the rate of convergence of an iterative method (in particular a Krylov space method) is closely related to the condition number of the preconditioned coefficient matrix. A good preconditioner $B$ for $Ax=b$ should satisfies: 
\begin{itemize}
\item The condition number $\kappa(BA)$ of the preconditioned system is small compared with $\kappa(A)$;
\item The action of $B$ on any $v \in V$ is computationally cheap and has good  parallel scalability.
\end{itemize}
A powerful tool for constructing and analyzing (multilevel) preconditioners and iterative methods is the method of (successive and parallel) subspace corrections. A systematic analysis of subspace correction methods for SPD problems has been introduced by \cite{Xu.J1992a}. Here we give a brief review of method of subspace corrections.

Let $A_i : V_i \rightarrow V_i$ be the restriction of $A$ on the subspace $V_i$, i.e.,
$$
(A_i u_i, v_i) = (A u_i, v_i), \qquad \forall u_i, v_i \in V_i.
$$
Assume that $Q_i : V \rightarrow V_i$ is the orthogonal projection with respect to the $L^2$-inner product, namely,
$$
(Q_i u, v_i) = (u, v_i), \qquad \forall v_i \in V_i.
$$
In a similar manner, we define the projection with respect to the $A$-inner product, i.e.,
$$
(P_i u, v_i)_A = (u, v_i)_A, \qquad \forall v_i \in V_i.
$$

For each $1 \le i \le N$, we introduce a SPD operator $S_i:V_i \rightarrow V_i$ that is an approximation of the inverse of $A_i$ such that
\begin{equation}\label{eqn:subspace}
\| I - S_i A_i \|_A < 1.
\end{equation} 
We can construct a successive subspace correction (SSC) method by generalizing the Gauss-Seidel iteration: Let $v=u^{m-1}$ be the current iteration and
\begin{equation}\label{eqn:ssc}
v = v + S_i Q_i (f-Av) \qquad i = 1, 2, \ldots, N.
\end{equation}
And the new iteration $u^m = v$. By denoting $T_i = S_i Q_i A : V \rightarrow V_i$ for each $i = 1:N$, we get
$$
u- u^m = (I-T_N)(I-T_{N-1})\cdots(I-T_1) (u-u^{m-1}).
$$ 
For simplicity we often define the successive subspace correction operator $B_{\text{SSC}}$ implicitly as follows
\begin{equation}\label{eqn:ssc2}
I - B_{\text{SSC}} A = (I-T_N)(I-T_{N-1})\cdots(I-T_1).
\end{equation}

The convergence analysis of SSC has been carried out by several previous work and a sharp estimate of the convergence rate has been originally given by \cite{Xu.J;Zikatanov.L2002}:
\begin{theorem}[X-Z Identity]\label{thm:xz}\rm
If \eqref{eqn:decomp} and \eqref{eqn:subspace} hold, then the successive subspace correction method \eqref{eqn:ssc} converges and the following identity holds:
$$
\|I - B_{\text{SSC}} A\|_A^2 = 1 - \frac{1}{C}, 
$$
where the non negative constant
$$
C = \sup_{\|v\|_A=1} \inf_{\sum_{i=1}^N v_i = v} \sum_{i = 1}^N \Big\| \overline T_i ^{-\frac{1}{2}}(v_i + T_i^* P_i \sum_{j>i} v_j) \Big\|_A^2 \quad \text{and} \quad \overline T_i = T_i + T_i^* - T_i^*T_i.
$$
\end{theorem}

\begin{remark}[Exact Solver for Subspace Correction]\rm
A common choice of the subspace solver is $S_i = A_i^{-1}$, i.e. the problems on subspaces $V_i$ are solved exactly. In this case, the constant in Theorem~\ref{thm:xz}
$$
C = \sup_{\|v\|_A=1} \inf_{\sum_{i=1}^N v_i = v} \sum_{i = 1}^N \Big\| P_i (\sum_{j \ge i} v_j) \Big\|_A^2.
$$
This identity has been utilized to analyze convergence rate of the multigrid methods and the domain decomposition methods.
\end{remark}

\begin{remark}[Parallel Subspace Correction]\rm
The operator $B_{\text{SSC}}$ in \eqref{eqn:ssc2} is often used as a preconditioner of the Krylov methods. An additive version of subspace correction method, the so-called parallel subspace correction (PSC) preconditioner, can be defined as 
\begin{equation}\label{eqn:psc}
B_{\text{PSC}} := \sum_{i=1}^N S_i Q_i.
\end{equation}
The preconditioned system 
$$
B_{\text{PSC}} A = \sum_{i=1}^N S_i Q_i A = \sum_{i=1}^N T_i.
$$
This type of preconditioners is often used for parallel computing as all the subspace solvers can be carried out independently and simultaneously, which is clear from the above equation. 
\end{remark}

\begin{remark}[Colorization]\label{rem:color}\rm
For parallel implementation of SSC, we need to employ colorization: Suppose we partition the computational domain into $N_{\cal C}$ colors, i.e., $D = \bigcup_{t=1}^{N_{\cal C}} \bigcup_{i\in {\cal C}(t)} D_i$ such that, for any $t=1,2,\ldots,N_{\cal C}$,
$$
P_i P_j = 0 \qquad \forall \, i,j \in \cup_{i\in {\cal C}(t)} D_i.
$$
Namely, $P_i$ and $P_j$ are orthogonal to each other if they belong to the same color $t$. This makes the parallelization among the same color possible. In this sense, SSC can be written as several successive PSC iterations using colorization:
$$
v = v + \sum_{i \in {\cal C}(t)} S_i Q_i (f-Av) \qquad t = 1,2,\ldots,N_{\cal C}.
$$
So we can use PSC as an example to demonstrate what will happen to subspace correction methods with presence of errors. This is because PSC is much easier to understand in the parallel setting. 
\end{remark}

\subsection{Parallel subspace correction in a faulty environment}

A special case of parallel subspace correction method is the widely-used classical additive Schwarz method~\parencite{toselli2005domain}. Here, as an example, we consider an overlapping version of the additive Schwarz method (ASM), which is often employed for large-scale parallel computers because of its efficiency and parallel scalability. A typical program flow chart of the additive Schwarz method in a not-error-free world (under the assumptions {\bf A1}--{\bf A3}) is given in Figure~\ref{fig:asm} (We use the Parallel Activity Trace graph or PAT by \cite{Deng2013} to denote the main ideas of the algorithms.\footnote{The y-axis is processing units and the x-axis is time. The solid bars stand for computational work and springs stand for inter-process communication.}) 
\begin{figure}[h!!] %  figure placement: here, top, bottom, or page
   \centering
   \includegraphics[width=0.9\linewidth]{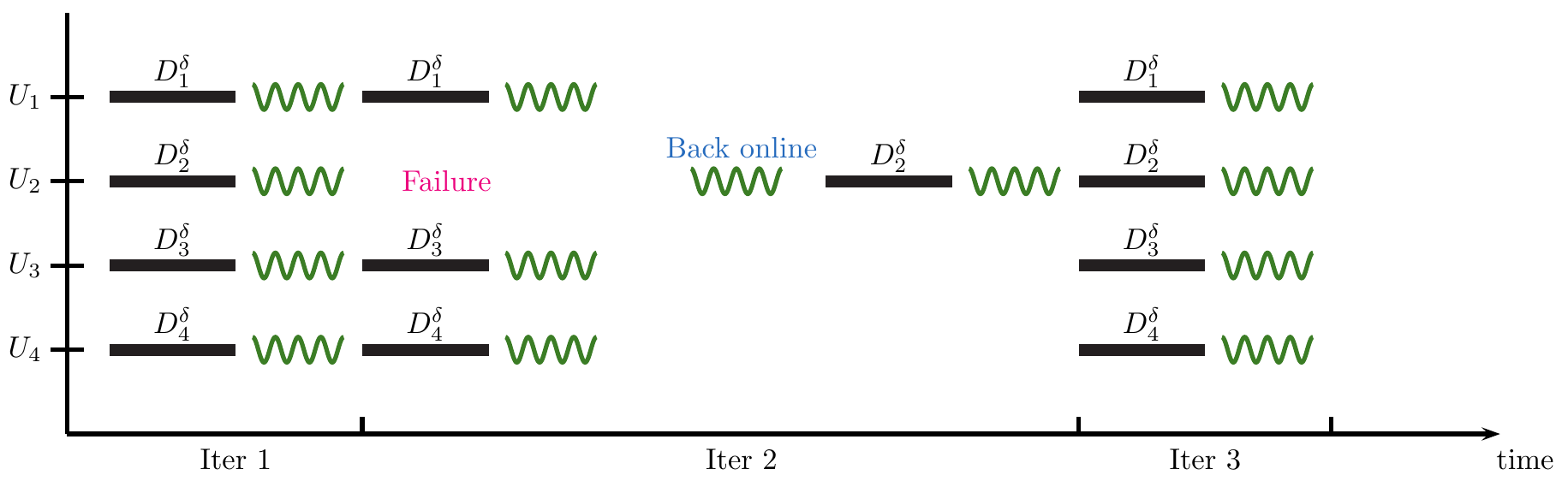} 
   \caption{Parallel subspace correction without error resilience}
   \label{fig:asm}
\end{figure}
When the processing unit $U_2$ fails to respond, the other processing units will be forced to wait until $U_2$ has been put back online; see, for example Iteration 2, in Figure~\ref{fig:asm}. Apparently this is not efficient as the processing unit could be offline for arbitrary length of time; see Assumption {\bf A3}. 

%\begin{remark}\rm
%Since $S_i$ is a SPD, the operator $B_{\text{PSC}}$ is also SPD and the $B_{\text{PSC}}$-preconditioned conjugate gradient method converges for solving $Au=f$; see Lemma 3.1~\parencite{Xu.J1992a}.
%\end{remark}

%The classical additive Schwarz method reads
%
%\begin{equation}\label{eqn:ASM} 
%B_{\text{AS}} := R^\delta_1 A_1^{-1} R^\delta_1 + \cdots + R^\delta_N A_N^{-1} R^\delta_N, 
%\end{equation}
%
%where $R^\delta_i$ is a $n \times n$ matrix which has a identity sub-block corresponding to $D^\delta_i$. Correspondlngly, the restrictive additive Schwarz method~\parencite{Cai1999} can be written as
%
%\begin{equation}\label{eqn:RASM}
%B_{\text{RAS}} := R^0_1 A_1^{-1} R^\delta_1 + \cdots + R^0_N A_N^{-1} R^\delta_N.
%\end{equation}

When $\delta$ is large enough, we can introduce a naive approach which makes use of the redundancy introduced by the overlaps and allows each processing unit to carry extra information from neighboring processing units. On the processing unit $i$, we use the redundant information in the overlapping region $D^{\delta-\gamma_i}_i \setminus D^0_i$ (buffer zone), when the processing unit who owns these DOFs fails. Here, $0 \le \gamma_i \le \delta$ and is usually not equal to 0 to reduce boundary pollution effects. 
\begin{figure}[h!!] %  figure placement: here, top, bottom, or page
   \centering
   \includegraphics[width=0.85\linewidth]{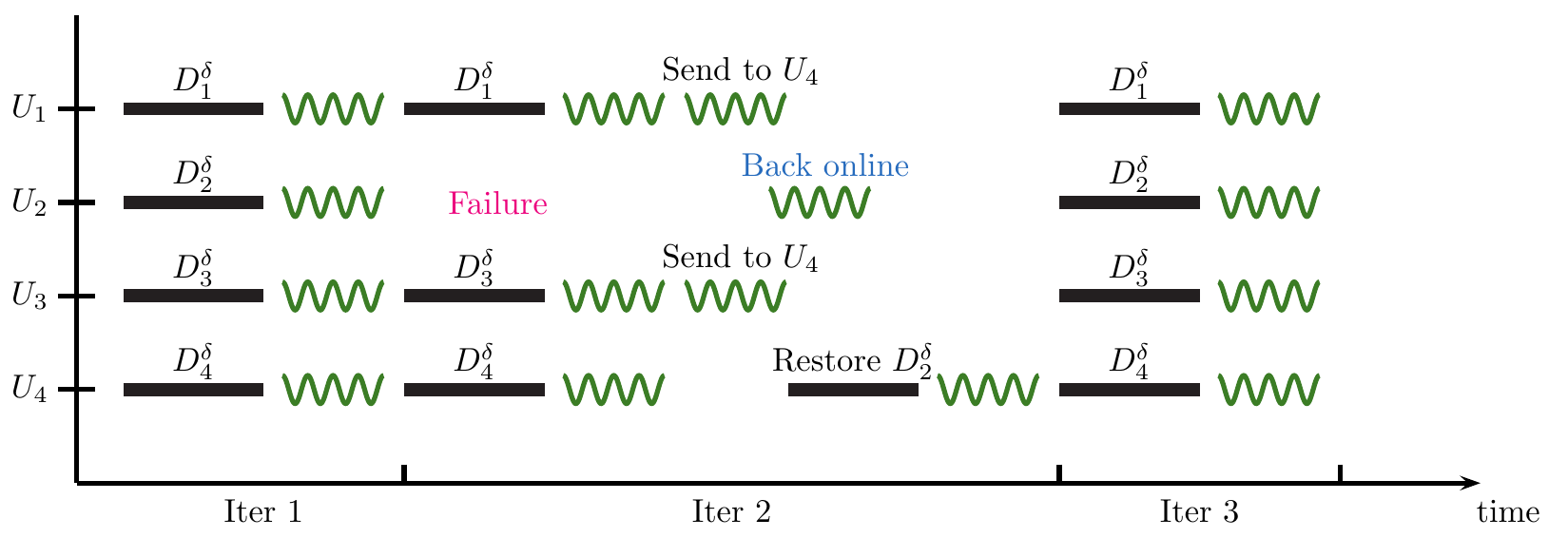} 
   \vskip -0.5cm
   \caption{Parallel subspace correction using data in $\delta$-overlapping areas to recover lost data}
   \label{fig:asm_buff}
\end{figure}
As an example, the union of the buffer zone on $U_1$ and $U_3$ could cover part of the degree of freedoms in $D_2$. When $U_2$ fails, we can request data for preconditioner as well as iterative method from $U_1$ and $U_3$; see Figure~\ref{fig:asm_buff}.

%Suppose the computer processing unit $j$ fails during the whole simulation. Let 
%
%$$
%\gamma_i := \left\{
%\begin{array}{ll}
%\delta, & \text{if } D^\delta_i \cap D^\delta_j = \varnothing \\
%\text{an integer} \le \delta, & \text{otherwise.}
%\end{array}
%\right.
%$$
%
%This method can be called the Fault Tolerant Parallel Subspace Correction (FTPSC) method and it can be written as
%
%\begin{equation}\label{eqn:FTPSC}
%B_{\text{FT}} := \sum_{i \neq j} R^{\delta-\gamma_i}_i A_i^{-1} R^\delta_i.
%\end{equation}
%

Due to the pollution effect, the convergence rate of this method deteriorates when there are failed processing units. It is easy to see that the approach discussed above is not realistic and it requires to introduce enough redundancy in order to achieve error resilience. 

%\begin{remark}\rm
%A special case of $B_{\text{FT}}$ is to ignore the preconditioning data in the failed processing unit $j$ (set $\gamma_j=\delta$) and replace it by an identity mapping. In this special case, we do not recover remote data and are able to improve parallel efficiency by using
%
%\begin{equation}\label{eqn:Local_FTPSC}
%B_{\text{LFT}} := R^0_j R^\delta_j + \sum_{i \neq j} R^0_i A_i^{-1} R^\delta_i.
%\end{equation}
%
%This method minimize the communication cost of $B_{\text{FT}}$. Due to this reason, we may call this method the Local Fault Tolerant Parallel Subspace Correction method or LFTPSC. On the other hand, this simplification could lead to a severe deteriorated convergence rate. 
%\end{remark}

%FTPSC is only of theoretical interest as the redundancy in the overlapped region is usually not enough to recover the lost information when a processing unit fails.  Although losing some information is not critical for the preconditioning procedure, it is critical for the matrix-vector product subroutine. 

%
%!TEX root = main.tex

\section{Method of redundant subspace corrections}\label{sec:msc}

In the previous section, we have discussed the behavior of method of subspace corrections (MSC) in a non-error-free environment. There are several possible ways to improve resilience of MSC and the key is to introduce redundancy. In fact, if we review the decomposition \eqref{eqn:decomp}, there is nothing to prevent us from repeating the subspaces $V_i$'s---We can have same subspace $V_i$ multiple time on different processing units. 

\subsection{Redundant subspaces}\label{sec:redundant}

One simple approach to introduce redundancy is to use multiple processes to solve each subspace problem. This is in the line of process duplication approach which is often used to enhance reliability of important and vulnerable components of an application. However this approach associates with high computation/communication overhead and shall not be applied for the whole system.

We now introduce another approach: We pair processing units and each processing unit carries its own data as well as the data for its brother (in the same pair) as redundancy information. We use a simple example to explain the main idea: We keep two distinctive subspaces in each processing unit as illustrated by the following distribution for a simple 4-subspace on 4-process case in Figure~\ref{fig:partition}:
$$
\begin{array}{ccc}
\text{Process} & \text{Owned Subspace} & \text{Redundant Subspace} \\
U_1 & V_1, D_1^\delta & V_2, D_2^\delta \\
U_2 & V_2, D_2^\delta & V_1, D_1^\delta \\
U_3 & V_3, D_3^\delta & V_4, D_4^\delta \\
U_4 & V_4, D_4^\delta & V_3, D_3^\delta
\end{array}
$$
where $U_i$ is the $i$-th processing unit and $V_j$ is the $j$-th subspace. 
Suppose $U_1$ has its owned subspace data $D^\delta_1$; in addition, it also has the data for $D^\delta_2$; see Figure~\ref{fig:partition} (Right). This way, when one processing unit ($U_2$) fails, its subspace solver $S_2$ can be carried out on the corresponding redundant processing unit ($U_1$). 
\begin{figure}[h] %  figure placement: here, top, bottom, or page
   \centering
   \includegraphics[width=0.85\linewidth]{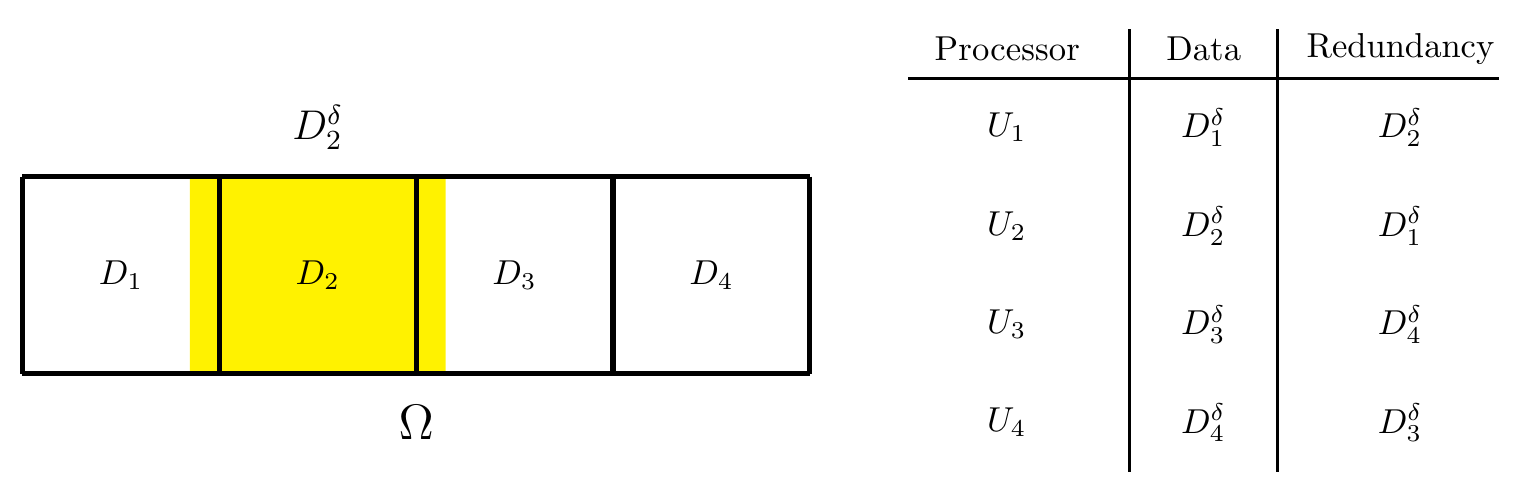} 
   \vskip -0.5cm
   \caption{Partition of the physical domain and redundant data storage}
   \label{fig:partition}
\end{figure}

Algorithmically, if we solve $D^\delta_2$ subproblem without using the solution in $D_1$ which has been calculated on $U_1$, then this method is equivalent to the classical additive Schwarz method; see Figure~\ref{fig:asm_rec}.\footnote{In this figure, we distinguish regular subspace and redundant subspace corrections by different colors.} An apparent drawback of this method is that, when one processing unit fails, the load balance of the parallel program is destroyed. 
\begin{figure}[h!!] %  figure placement: here, top, bottom, or page
   \centering
   \includegraphics[width=0.85\linewidth]{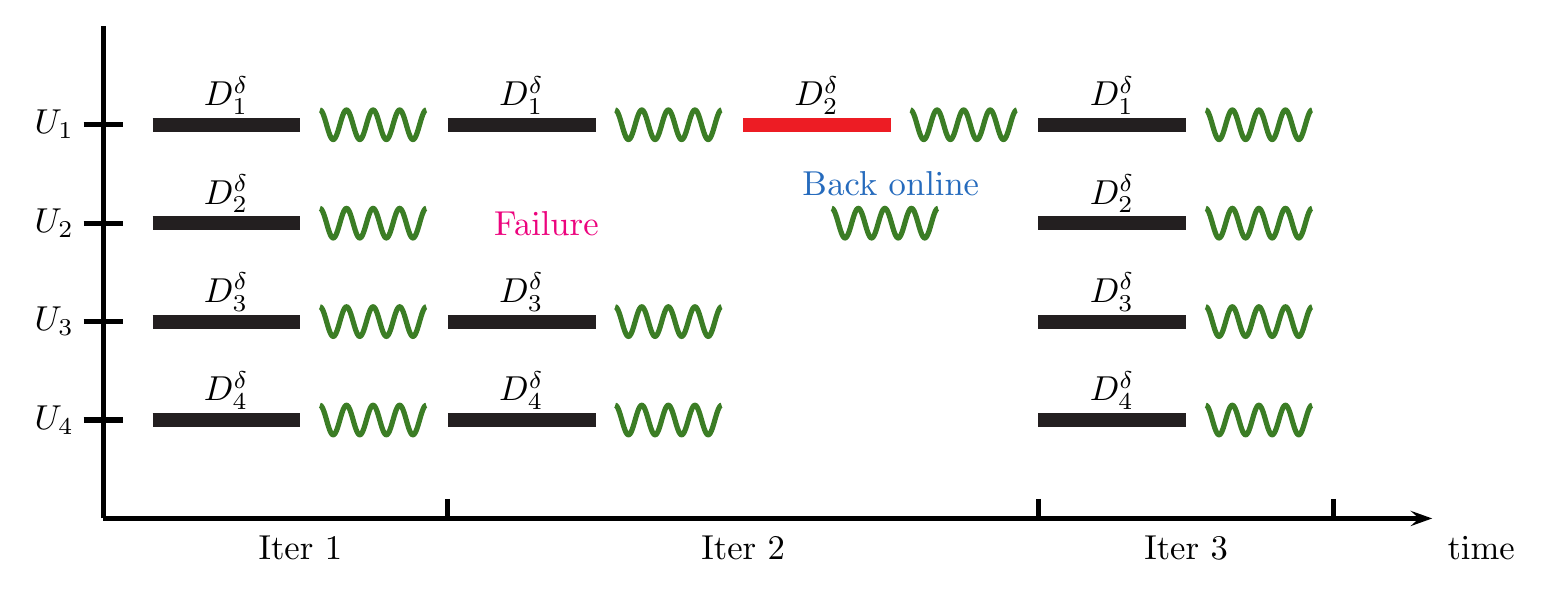} 
   \vskip -0.5cm
   \caption{Parallel subspace correction using redundant information to perform subspace solver for an erroneous processing unit}
   \label{fig:asm_rec}
\end{figure}

\begin{remark}[Subspace Corrections with Redundancy]\rm
When $U_2$ fails (see Figure~\ref{fig:asm_rec}), we can use the solution obtained in $D_1$ (because it is easily available) before we solve the subspace problem in $D_2^\delta$ and obtain a slightly better solution for $D_2$. This method in turn improves the convergence rate. However, it still causes most of the processing units to be idle during the erroneous states, which makes the method not desirable. 
\end{remark}

\subsection{Compromised redundant subspace corrections}\label{ssc:full}

To improve load balance of the method in \S\ref{sec:redundant} (as illustrated in Figure~\ref{fig:asm_rec}) in massively parallel environment, we choose to use a computationally cheap subspace solver $S_j^c$ instead of $S_j$ for the erroneous processing unit $j$. 

We consider the same example as in \S\ref{sec:redundant}. Assume that $U_2$ fails. We then have the following parallel subspace correction:
$$
\begin{array}{lll}
U_1: & V_1 & S_1 \\
U_1: & V_2 & S_2^c \\
U_3: & V_3 & S_3 \\
U_4: & V_4 & S_4 \\
\end{array}
$$
Here, $S_i$ is the usual (approximate) inverse or a preconditioner of the local matrix associated with subspace $V_i$. On the other hand, $S_j^c$ is a compromised subspace solver/preconditioner---This operator will be used to replace $S_j$ when the $j$-th processing unit fails and part or whole information of the subspace $V_j$ is not available. When a processing unit ($U_2$ for example) fails to return correct results, we could make use of the redundant subspace information (stored on $U_1$) for this erroneous process to recover the corresponding subspace solver results. 

The compromised subspace solver $S_j^c$ can be simply a proper scaling $\alpha_j I$, where $\alpha_i$ is a positive scaling parameter. In fact, it is equivalent to replace the exact subspace solver by the Richardson method for the subspace problem on $V_j$. Of course, we can also choose to use weighted Jacobi method instead. We now arrive at the following iterative scheme: Replacing the iterative method \eqref{eqn:ssc} in SSC by
\begin{eqnarray}
v &=& v + S_i Q_i (f-Av) \qquad i = 1, 2, \ldots, j-1 \label{eqn:rssc1}\\
v &=& v + S_j^c Q_j (f-Av) \label{eqn:rssc2} \\
v &=& v + S_i Q_i (f-Av) \qquad i = j+1, \ldots, N.\label{eqn:rssc3}
\end{eqnarray}
This yields the compromised redundant subspace correction method 
\begin{equation}\label{eqn:rssc}
I - B_{\text{SSC}}^c A = (I-T_N)\cdots(I-T_{j+1})(I-T_j^c)(I-T_{j-1})\cdots(I-T_1).
\end{equation}

By choosing $S_j^c = \alpha_j I$, we have 
$$
T_j^c = S_j^c Q_j A = \alpha_j Q_j A = \alpha_j A_j P_j.
$$
It is easy to see that, if $\alpha_j$ is small enough, then $\| I - S_j^c A_j \|_A = \| I - \alpha_j A_j \| < 1$ and $\overline T_j^c$ is symmetric positive definite (\cite[Lemma 4.1]{Xu.J;Zikatanov.L2002}). We can then obtain the following convergence result using  Theorem~\ref{thm:xz}:

\begin{corollary}[Convergence of Compromised Redundant Subspace Corrections]\rm
If the $j$-th processing core is in the erroneous state and $\alpha_j$ is small enough, $\|I - B_{\text{SSC}}^c A\|<1$. Hence the iterative method \eqref{eqn:rssc1}--\eqref{eqn:rssc3} converges. 
\end{corollary}

\begin{remark}[Residual Computation]\rm
The coefficient matrix $A$, the solution vector $v$, and the right hand side $f$ are stored in distributed memory model with redundancy. The residual $r=f-Av$ can be computed by the redundant data when an error or failure is captured. On the 4-process case as in Figure~\ref{fig:asm_rec}, 
$A=(A_1^T, A_2^T, A_3^T, A_4^T)^T$,
 $v=(v_1, v_2, v_3, v_4)^T$ and $f=(f_1, f_2, f_3, f_4)^T$
are stored as:
$$
\begin{array}{ccc}
\text{Process} & \text{Owned Data} & \text{Redundant Data} \\
U_1 & A_1, v_1, f_1 & A_2, v_2, f_2 \\
U_2 & A_2, v_2, f_2 & A_1, v_1, f_1 \\
U_3 & A_3, v_3, f_3 & A_4, v_4, f_4\\
U_4 & A_4, v_4, f_4 & A_3, v_3, f_3
\end{array}
$$
Subspace data $A_i$ and $f_i$ remain the same in each iteration and the redundant $v_i$ (e.g. $v_1$ on $U_2$) must be updated when owned $v_i$ (e.g. $v_1$ on $U_1$) is changed and vice versa. This introduces an extra point-to-point communication (in each processor pair). When there is an error or failure captured on $U_2$, we can use the redundant $A_2$, $v_2$ and $f_2$  stored on $U_{1}$ to compute the residule vector which requires one matrix-vector operation and one vector-vector operation on  $U_1$ .
\end{remark}

\subsection{Improving parallel scalability and efficiency}\label{sec:improve}

We have introduced a new subspace correction method with redundant information above. However, this approach is not desirable as all processing units except $U_1$, when it carries out the subspace solver $S_1^c$. Even though $S_j^c$ is much cheaper than the usual subspace solver $S_j$, it still cause undesirable idle for the majority of the processing units. In this subsection, we discuss how to improve parallel scalability and efficiency of the compromised redundant subspace correction method \eqref{eqn:rssc1}--\eqref{eqn:rssc3}. 

In order to remove this idle part of the algorithm completely, we choose $S_j^c = 0$ in the compromised redundant subspace correction method, i.e.,
\begin{eqnarray}
v &=& v + S_i Q_i (f-Av) \qquad i = 1, 2, \ldots, j-1 \\
v &=& v + S_i Q_i (f-Av) \qquad i = j+1, \ldots, N.
\end{eqnarray}
We use the example in Figure~\ref{fig:partition} to demonstrate the idea. In this case the iteration operator 
\begin{equation}\label{eqn:srsc1}
I - B^c_{\text{SSC}} A = (I - T_4) (I - T_3) (I - T_1).
\end{equation}
Of course this method will not be reliable as one the subspace never been corrected when the process is erroneous. This is because we completely ignore the redundant information.

Now we add another iteration step to compensate the loss information with the help of the redundant subspace to make another ``compromised'' subspace correction using
$$
\begin{array}{lll}
U_1: & V_2 & S_2 \\
U_3: & V_4 & S_4 \\
U_4: & V_3 & S_3
\end{array}
$$
This gives another iteration operator:
\begin{equation}\label{eqn:srsc2}
I - \tilde B^c_{\text{SSC}} A = (I - T_3) (I - T_4) (I - T_2).
\end{equation}

We then have the successive redundant subspace correction (SRSC) method
\begin{equation}\label{eqn:srsc}
I - B_{\text{SRSC}} A = (I - \tilde B^c_{\text{SSC}} A)(I - B^c_{\text{SSC}} A).
\end{equation}
See the flow chart in Figure~\ref{fig:rsc} for an illustration. 
\begin{figure}[htbp] %  figure placement: here, top, bottom, or page
   \centering
   \includegraphics[width=0.9\linewidth]{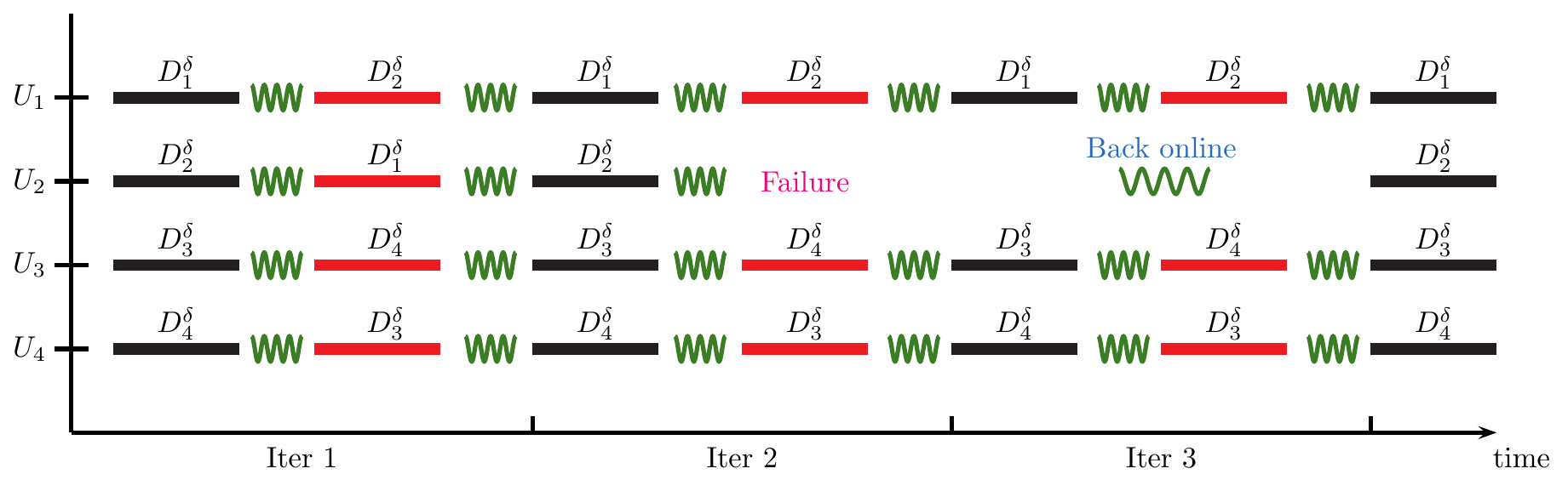} 
   \vskip -0.5cm
   \caption{Redundant subspace correction method}
   \label{fig:rsc}
\end{figure}  

\begin{remark}[Error/Failure Handling]\label{rem:error-handling}\rm
We now consider error and failure handling in the virtual machine environment discussed in \S\ref{sec:virtual}. In the redundant subspace correction method, when errors are detected in a process, we directly put this process into the failed state and take it out from the redundant subspace correction iteration. After the error on that process has been corrected, we recover this process from the failed state and resynchronize it with other processes for the iterative procedure. This error handling can also be applied for a fail-stop process caused by non-responsive nodes, which makes local failure local recovery (LFLR) possible. 
\end{remark}

\begin{remark}[Overhead of RSC]\rm
The main idea of RSC is that, by locally keeping redundant subspaces in appropriate processing units, lost information can be retrieved from the redundant subspaces to keep the iterative method as well as the preconditioning procedure to continue without compromising convergence rate when failure of some processing threads or computing processing units occurs. The overhead in computing work and communication is marginal when no failure occurs. 
\end{remark}

\begin{remark}[SRSC When Error-Free]\label{rem:error-free}\rm
We can see that the convergence rate of SRSC is at least as good as the corresponding SSC method in the worse case scenario. In fact, if there is no error occurs, then the identity \eqref{eqn:srsc} yields that 
$$
I - B_{\text{SRSC}} A = (I - \tilde B_{\text{SSC}} A)(I - B_{\text{SSC}} A),
$$
i.e., the SRSC method converges twice as fast as the corresponding SSC method. 
\end{remark}

\begin{theorem}[Convergence Estimate of Redundant Subspace Correction]\label{thm:srsc}\rm
If an error occurs during computation, the convergence rate of the successive redundant subspace correction method \eqref{eqn:srsc} satisfies 
\begin{equation}\label{thm:1}
\|I - B_{\text{SRSC}} A\|_A \le \|I - B_{\text{SSC}} A\|_A. 
\end{equation}
If there is no error during computation, the convergence rate satisfies that
\begin{equation}\label{thm:2}
\|I - B_{\text{SRSC}} A\|_A \le \|I - B_{\text{SSC}} A\|_A \; \|I - \tilde B_{\text{SSC}} A\|_A. 
\end{equation}
\end{theorem}
\begin{proof}
With loss of generality, we assume that the processing unit which contains the data for the subspace $V_1$ (and $V_2$ as the redundant subspace) fails or is taken out of the iteration due to errors. Let $W_i = V_i$ if $1 \le i \le N$, and $W_i = V_{i-N+2}$ if $N < i \le 2N-2$. In this case, we have the space decomposition 
$$
V = \sum_{i=1}^N V_i + \sum_{k=3}^N V_k = \sum_{i=1}^{2N-2} W_i,
$$
where $V_k$ ($k=3, \ldots, N$) are the redundant subspaces. 
For any $v \in V$, we have a decomposition
$$
v = \sum_{i=1}^{2N-2} v_i \quad \text{and} \quad v_i \in W_i \; (i=1,\ldots,2N-2).
$$
Moreover, we have a special case of this decomposition is 
$$
v = \sum_{i=1}^{2N-2} w_i = \sum_{i=1}^N w_i, \qquad w_i  \in W_i.
$$ 
In another word, $w_i = 0$ if $N < i \le 2N-2$. We then immediately obtain that
$$
\inf_{v=\sum_{i=1}^{2N-2}v_i} \sum_{i=1}^{2N-2} \|\overline T_i ^{-\frac{1}{2}} \Big( v_i + T_i^* P_i \sum_{j>i} v_j \Big) \|_A ^2 
\le \sum_{i=1}^{N} \|\overline T_i ^{-\frac{1}{2}} \Big( w_i + T_i^* P_i \sum_{N \ge j>i} w_j \Big) \|_A ^2.
$$
As $w_i \in W_i = V_i$ ($i=1,2,\ldots,N$) could be anything, we have 
$$
\inf_{v=\sum_{i=1}^{2N-2}v_i} \sum_{i=1}^{2N-2} \|\overline T_i ^{-\frac{1}{2}} \Big( v_i + T_i^* P_i \sum_{j>i} v_j \Big) \|_A ^2 
\le \inf_{v=\sum_{i=1}^{N}v_i} \sum_{i=1}^{N} \|\overline T_i ^{-\frac{1}{2}} \Big( v_i + T_i^* P_i \sum_{j>i} v_j \Big) \|_A ^2.
$$
The inequality \eqref{thm:1} of the theorem then follows from the above inequality and Theorem~\ref{thm:xz}. The equality \eqref{thm:2} is straightforward from Remark~\ref{rem:error-free}.
\end{proof}

\begin{remark}[More Erroneous Processing Units]\rm
Although we assume only one processing unit can be in the erroneous state (Assumption~{\bf A1}), we can easily see, from Theorem~\ref{thm:srsc}, that the method still converges as long as at least one processing unit from each pair works correctly.
\end{remark}

The corresponding preconditioner of the parallel subspace correction method \eqref{eqn:psc} can be written as follows:
\begin{equation}
B_{\text{PSC}}^c:=S_1 Q_1+S_3 Q_3+S_4 Q_4.
\end{equation}
Using a similar approach as in SRSC, we then apply a parallel subspace correction from the redundant copy of subspace preconditioner to make another ``compromised'' subspace correction using
\begin{equation}
\tilde B_{\text{PSC}}^c := S_2 Q_2+S_4 Q_4+S_3 Q_3.
\end{equation}
Finally, we combine the above two incomplete subspace correction preconditioners, $B_{\text{PSC}}^c$ and $\tilde B_{\text{PSC}}^c$, in a multiplicative fashion to obtain a new preconditioner $B_{\text{PRSC}}$:
$$
I-B_{\text{PRSC}} A=(I - \tilde B_{\text{PSC}}^c A) (I - B_{\text{PSC}}^c A). 
$$
This is an example of the Redundant Subspace Correction (RSC) method; see Figure~\ref{fig:rsc}.

\begin{comment}
To guarantee the effectiveness of this preconditioner $B_{\text{PRSC}}$, we need to
\begin{itemize}
\item introduce local scalings so that both $I-B_{\text{PSC}}^c A$ and $I-\tilde B_{\text{PSC}}^c A$ are non-expansive, i.e.,
$$\|(I-B_{\text{PSC}}^c A)v\|_A \le \|v\|_A \quad \text{and} \quad \|(I-\tilde B_{\text{PSC}}^c A)v\|_A \le \|v\|_A, \qquad \forall v \in V;$$
\item introduce a symmetrization such that the preconditioner is symmetric positive definite.
\end{itemize}
\end{comment}

\begin{remark}[PRSC When Error-Free]\rm
If we use a nested sequence of subspaces $V_1 \subset V_2 \subset \cdots \subset V_N \equiv V$, then the method is actually the BPX preconditioner~\parencite{Bramble.J;Pasciak.J;Xu.J1990}. When no error occurs during the iterative procedure, we have
$$I - B_{\text{PRSC}} A = (I - B_{\text{PSC}} A)^2 = (I - B_{\text{BPX}} A)^2.$$
\end{remark}

%
%!TEX root = main.tex

\section{Numerical Experiments}\label{sec:numerics}

In this section, we design a few numerical experiments to test the proposed redundant subspace correction methods with a few widely used partial differential equations and their standard discretizations.

\subsection{Test problems}\label{ss:tests}

The numerical experiments are done for the Poisson equation, the Maxwell equation, and the linear elasticity equation in three space dimension with the Dirichlet boundary condition. The computational domain is the unit cube $\Omega=(0,1)^3$. The domain partitioning has been done using the METIS package~\parencite{karypis1998fast} and a sample partition is given in Figure~\ref{fig:ddm}.
\begin{figure}[htbp] %  figure placement: here, top, bottom, or page
   \centering
   \includegraphics[width=0.45\linewidth]{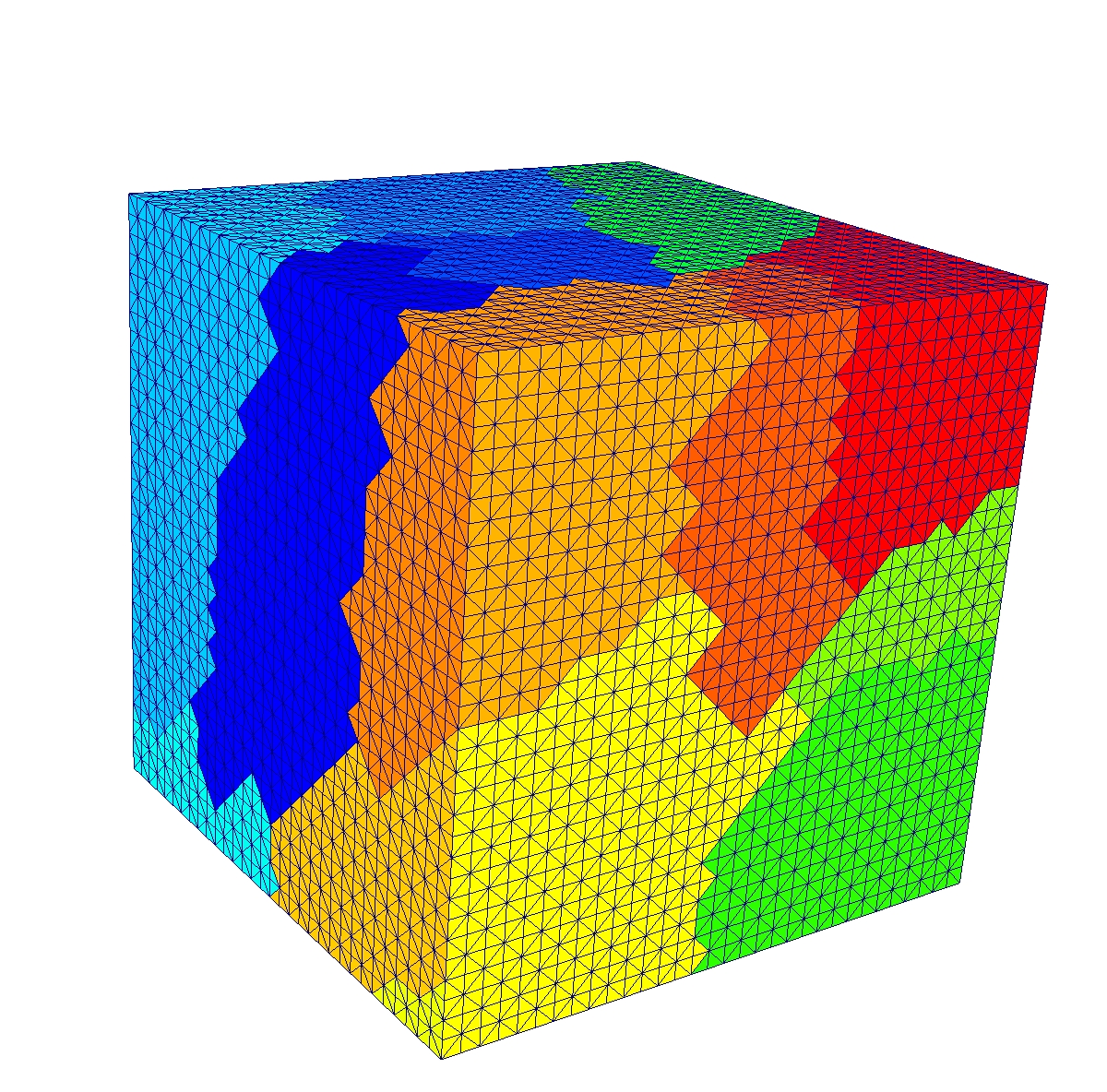} 
   \vskip -0.5cm
   \caption{A sample domain partition of a unit cube for the Poisson equation}
   \label{fig:ddm}
\end{figure}

\example{The Poisson's equation}\rm
\begin{equation}
  \begin{cases}
    - \Delta u = f,& \mbox{ in } \Omega \\
      u = g,& \mbox{ on }\partial\Omega
  \end{cases}
\end{equation}
The first order lagrange element is used for discretization. We use the continuous piecewise linear Lagrange finite element (FE) discretization to solve this equation. 

\example{The Maxwell equation}\rm
\begin{equation}
  \begin{cases}
      \nabla \times \mu^{-1} \nabla \times\vec{E} - k^2 \vec{E} = \vec{J}, & \mbox{ in } \Omega \\
      \vec{E} \times \vec{n} = \vec{g} \times \vec{n}, &\mbox{ on }\partial\Omega
  \end{cases}
\end{equation}
The parameters $\mu=1$ and $k^2=-1$. The exact solution is chosen to be
\begin{equation*}
\left(
\begin{array}{c}
  xyz(x-1)(y-1)(z-1)(x-0.5)(y-0.5)(z-0.5) \\
\sin(2\pi x)\sin(2\pi y) \sin(2\pi z) \\
(1-e^x)(e-e^x)(e-e^{2x})(1-e^y)(e-e^y)(e-e^{2y})(1-e^z)(e-e^z)(e-e^{2z})
\end{array}
\right).
\end{equation*}
The lowest order edge element is used for discretization. 

\example{The linear elasticity equation}\rm
\begin{equation}
  \begin{cases}
      \nabla \cdot \mathbf{\tau} = \vec{f}, & \vec{x} \in \Omega \\
      \vec{u} = \vec{g}, & \vec{x} \in \partial\Omega
  \end{cases}
\end{equation}
where
\begin{equation}
\tau_{ij} = 2\mu \epsilon_{ij} + \lambda \delta_{ij} \epsilon_{kk}, 
\quad
\epsilon_{ij} = \frac{1}{2}(u_{i,j}+u_{j,i})
\qquad (i,j=1,2,3),
\end{equation}
and
$u_{i,j}=\partial u_i/ \partial x_j$. The parameters are given by
\begin{equation}
\left\{
\begin{array}{c}
\lambda = \frac{E\nu}{(1+\nu)(1-2\nu)} \\
\mu = \frac{E}{2(1+\nu)},
\end{array}
\right.
\end{equation}
where 
$E=2.0$ and $\nu = 0.25$. The continuous piecewise quadratic Lagrange finite element is used for discretization. 

\subsection{Implementation details}

All numerical tests are carried out on the LSSC-I\!I\!I cluster at State Key Laboratory of Scientific and Engineering Computing (LSEC), Chinese Academy of Sciences. The LSSC-I\!I\!I cluster has 282 computing nodes: Each node has two Intel Quad Core Xeon X5550 2.66GHz processors and 24GB shared memory; all nodes are connected via Gigabit Ethernet and DDR InfiniBand. Our implementation is based on~\cite{phg}, which is a toolbox for developing parallel adaptive finite element programs on unstructured tetrahedral meshes and it is under active development at the LSEC. 

We use MPI distributed memory parallelism paradigm and a processing unit is just one core in a multicore cluster in our experiments. We simplify the non-error-free environment by setting one of the process to be fail and not responding from beginning to end of iterative methods. This way, the failed core does not contribute to the solution of linear systems at all. This removes the complication for considering detecting and fixing the error, which allows us to focus on the convergence and scalability of the proposed RSC methods. Furthermore, this also free us from considering the overhead introduced by detecting and fixing errors and we can obtain a good idea on the algorithmic overhead introduced by the error resilience feature of our algorithm. 

In the following of this section, we present a few preliminary numerical examples for the performance of the proposed methods on a virtual machine as discussed in \S\ref{sec:virtual}. We mainly interested in testing the following: (1)~convergence of the successive redundant subspace correction (SRSC) method as an iterative method; (2)~algorithmic overhead introduced by SRSC compared with regular SSC; (3)~performance of the parallel redundant subspace correction (PRSC) method as a preconditioner and its overhead; (4)~weak scalability of PRSC as a preconditioner. 

Since the preconditioner action might change during the iteration, we should use flexible versions of the Krylov space iterative methods together with PRSC, such as the Flexible Conjugate Gradient (FCG) or the Flexible Generalized Minimal Residual (FGMRES) method with restart. We employ the Flexible GMRES method~(\cite{Saad1996}) as the iterative solver and we need a resilient iterative method as well. In all our numerical experiments, FGMRES with restarting number $30$ is used and the maximum iteration number is set to be $10000$. One can consider to combine the FT-FGMRES~\parencite{Hoemmen} with the proposed redundant subspace correction preconditioners to improve convergence rate of sparse iterative solvers. 

In the numerical experiments, we choose an extensively studied algorithm, the domain decomposition method with out the coarse space, which can be analyzed as a special case of the method of subspace correction; see \cite{Chan1994,toselli2005domain} for a comprehensive overview of the field. We employ the multiplicative Schwarz method (a SSC method) and the additive Schwarz method (a PSC method) with overlapping level $\delta=2$.\footnote{Note that additive and multiplicative Schwarz methods with coarse mesh correction are not be the best options for the test problems under consideration; see more discussions in \S\ref{ss:scale}.} To make a fair comparison, we always start the iterative procedure from a zero initial guess in our tests. We terminate the iterative procedure when the relative Euclidean residual less than a fixed tolerance $tol=10^{-8}$. In the tables, ``\#Iter'' denotes the number of iterations, ``DOF'' denotes the degree of freedom, and ``Time'' denotes the wall time for computation in seconds.  

\subsection{Convergence and efficiency}\label{ss:convergence}

First we test the convergence of the proposed redundant subspace correction method (SRSC) and we are interested in the impact of one erroneous process. In this test, we use 16 processing cores and the results are reported in Table~\ref{tab:psc-solver}. In a non-error-free case, we let processing core $U_1$ fail from the starting till the end of computation as we mentioned earlier. 
%\begin{table}[h!!]
%   \centering
%   \begin{tabular}{@{} ccccccc @{}} 
%         \toprule
%\multirow{2}{*}{Error} & \multicolumn{2}{c}{Poisson}  & \multicolumn{2}{c}{Maxwell} &\multicolumn{2}{c}{Elasticity}  \\
%      \cmidrule(lr){2-3}\cmidrule(lr){4-5}\cmidrule(lr){6-7} 
%    & \#Iter & Time & \#Iter & Time & \#Iter & Time \\
%\hline
%No    &  37     &  37.09  &   74     &   43.05   & 12 & 67.8  \\
%\hline
%Yes   &  44     &  42.06  &  95   &  53.29   & 13 & 69.1  \\
%\bottomrule
%   \end{tabular}
%   \caption{Convergence of PRSC as an iterative method.}
%   \label{tab:psc-solver}
%\end{table}
\begin{table}[h!!]
   \centering
   \begin{tabular}{@{} ccccccc @{}} 
         \toprule
\multirow{3}{*}{Error-Free} 
& \multicolumn{2}{c}{Poisson}  & \multicolumn{2}{c}{Maxwell} &\multicolumn{2}{c}{Elasticity}  \\
& \multicolumn{2}{c}{(2,146,689 DOFs)}  & \multicolumn{2}{c}{(1,642,688 DOFs)} &\multicolumn{2}{c}{(823,875 DOFs)} \\
      \cmidrule(lr){2-3}\cmidrule(lr){4-5}\cmidrule(lr){6-7} 
    & \#Iter & Time & \#Iter & Time & \#Iter & Time \\
\hline
Yes    &  44     &  70.73  &   63    &   68.76   & 73 & 223.14  \\
\hline
No   &  48     &  81.01  &  67  &  74.28   & 74 & 229.21  \\
\bottomrule
   \end{tabular}
   \caption{Convergence of colorized SRSC as an iterative method in error-free and non-error-free environments}
   \label{tab:psc-solver}
\end{table}
From the numerical results, we find that the proposed SRSC method converges. Furthermore, even with ${1 \over 16}$ of the processes failed, the convergence rate of the method does not deteriorate much---Number of iterations increase by $9\%$ or less. This is exactly what we expect based on the theoretical estimates in \S\ref{sec:msc}.

Next we compare the performance of RPSC and the standard PSC method as a preconditioner of FGMRES when no error occurs and when error occurs. In this test, we use 16 processing cores and the results are reported in Table~\ref{tab:psc-overhead}. Here we use the additive Schwarz method with overlap $\delta=2$. In a non-error-free case, we let processing core $U_1$ fail from the starting till the end of computation. 
%
%\begin{table}[h!!]
%   \centering
%   \begin{tabular}{@{} cccccccc @{}} 
%         \toprule
%   \multirow{2}{*}{Example} & DOF &\multicolumn{2}{c}{Error-free $B_\text{PSC}$} & \multicolumn{2}{c}{Error-free $B_\text{PRSC}$} & \multicolumn{2}{c}{$B_\text{PRSC}$ with Error} \\
%      \cmidrule(lr){3-4}\cmidrule(lr){5-6}\cmidrule(lr){7-8} 
%   & & \#Iter      & Time      & \#Iter     &      Time   & \#Iter     &      Time \\
%       \hline
%   Poisson & 2,146,689 & 17 & 22.54 & 9 & 22.81 & 11 & 24.49 \\    
%       \hline
%   Maxwell & 1,642,686 & 44 & 21.67 & 21 & 21.53 & 25 & 23.12 \\    
%       \hline
%   Elasticity& 823,875 & 12 & 62.76 &  7 & 63.96 & 8 & 64.91 \\    
%      \bottomrule
%   \end{tabular}
%   \caption{Performance of parallel redundant subspace correction preconditioner in error-free and non-error-free environments}
%   \label{tab:psc-overhead}
%\end{table}
%
%
\begin{table}[h!!]
   \centering
   \begin{tabular}{@{} cccccccc @{}} 
         \toprule
   \multirow{2}{*}{Example} & \multirow{2}{*}{DOF} &\multicolumn{2}{c}{$B_\text{PSC}$ Error-Free} & \multicolumn{2}{c}{$B_\text{PRSC}$ Error-Free} & \multicolumn{2}{c}{$B_\text{PRSC}$ With Error} \\
      \cmidrule(lr){3-4}\cmidrule(lr){5-6}\cmidrule(lr){7-8} 
   & & \#Iter      & Time      & \#Iter     &      Time   & \#Iter     &      Time \\
       \hline
   Poisson & 1,335,489 & 23 &  7.92 & 12& 8.09  & 13 & 8.13  \\    
       \hline
   Maxwell &   468,064 & 42 & 4.09  & 21 &  4.23 & 24 & 4.48  \\    
       \hline
   Elasticity& 436,515 & 16 & 10.18 &  9 & 11.01 & 10& 11.35 \\    
      \bottomrule
   \end{tabular}
   \caption{Performance of parallel redundant subspace correction preconditioner in error-free and non-error-free environments}
   \label{tab:psc-overhead}
\end{table}
In Table~\ref{tab:psc-overhead} we notice that the overhead introduced by the redundant subspace correction method is small from two perspectives:
\begin{itemize}
\item When there is no error, the PRSC method is still efficient compared with the standard PSC method.
\item When there is error, the PRSC method converges and the extra cost in term of wall time is less than 10\% compared with the case when there is no error. 
\end{itemize}

\subsection{Weak scalability}\label{ss:scale}

Now we focus on weak scalability of the proposed method and compare the results in the error-free case with the case when the computation is affected by a single erroneous processing core. As before we use the additive Schwarz method with overlap $\delta=2$. It is well-known that the additive Schwarz method yields a preconditioner $B_\text{AS}$ whose performance deteriorates as the size of subdomains $H$ decreases. More precisely, if $\beta$ is the ratio between the size of the overlapping region and $H$, then the condition number of the preconditioned system
$$
\kappa (B_\text{AS} A) \le C H^{-2} (1 + \beta^{-2}),
$$ 
where the constant $C$ is independent of the mesh size $h$ or $H$~\parencite{Dryja.DryjaWidlund.1989fk,dryja1992additive}. This drawback can be fixed by introducing coarse grid corrections, which in turn requires a global communication of information and needs careful implementation~\parencite{gropp1992parallel,bjorstad1992domain,smith1993parallel}. 

Because we only wish to examine the impact of redundant subspace corrections, the Schwarz methods without coarse grid corrections are good enough for this purpose. The number of iterations, wall time in seconds, and parallel efficiency are reported in Tables~\ref{tab:poisson-scal}, \ref{tab:maxwell-scal}, and \ref{tab:els-scal}. From these experimental results, we can see that the PRSC method is robust if there is one failed processing core. Furthermore, the weak scalability of the preconditioner is reasonable and it is not contaminated much by the presence of failed processes. Note that the low parallel efficiency is mainly due to the fact that the method itself is not optimal and number of iterations increases as the mesh size decreases. 

\begin{table}[htbp]
   \centering
   \begin{tabular}{@{} rrcccccc @{}} 
         \toprule
\multirow{2}{*}{DOF} & \multirow{2}{*}{\#Cores}  & \multicolumn{3}{c}{Error-Free} &  \multicolumn{3}{c}{With Error} \\
                                \cmidrule(lr){3-5} \cmidrule(lr){6-8}
            &     & \#Iter   & Time  & Efficiency & \#Iter       &     Time  &Efficiency \\
\hline
536,769       & 8     &   8      & 5.09    & 100\% &   10 &   5.51   & 100\%  \\
1,335,489    & 16   &   12    & 8.09  &  62.9\% &   13 &   8.13   &  67.8\% \\
2,146,689    & 32   &   13    & 8.64   &  58.9\% & 15 &   8.99   & 61.3\% \\
4,243,841     & 64   &   14    & 8.91   &  57.1\% & 16 &   9.37   &  58.8\% \\
10,584,449   & 128 &   19    & 12.87  &  49.5\% &  20 &   13.95   &  39.5\% \\
16,974,593   & 256 &   23    & 18.01  &  28.3\% &  25 &   19.13   &  28.8\% \\
33,751,809   & 512 &   25    & 20.90   &  24.3\% & 27 &   26.11   &  21.1\% \\
      \bottomrule
   \end{tabular}
   \caption{Performance of the PRSC preconditioner for the Poisson equation}
   \label{tab:poisson-scal}
\end{table}

\begin{table}[htbp]
   \centering
   \begin{tabular}{@{} rrcccccc @{}} 
         \toprule
\multirow{2}{*}{DOF} & \multirow{2}{*}{\#Cores}  & \multicolumn{3}{c}{Error-Free} &  \multicolumn{3}{c}{With Error} \\
                                \cmidrule(lr){3-5} \cmidrule(lr){6-8}
            &   & \#Iter   & Time  & Efficiency & \#Iter       &     Time  &Efficiency \\
\hline
 238,688    & 8   &  15      & 4.08   & 100\% & 17 &  4.48    &  100\% \\
 468,064    & 16  &  21      & 4.23  &  96.5\% & 24 &  4.88   &  91.8\% \\
 968,800    & 32  &  23      & 5.18   & 78.8\% & 26 &  5.46    &  82.1\% \\
 1,872,064    & 64  &  27      & 7.21   &  56.6\% & 30 &  8.16   &  59.8\% \\
 3,707,072    &128  &  49      & 8.02   &  50.9\% & 54 &  8.84   & 54.9\% \\
 7,676,096    &256  &  51      & 10.60  &  38.5\% & 56 &  11.99   & 37.4\%\\
14,827,904   &512  &  65      & 17.67  &  23.1\% & 73 &  19.52   & 23.0\% \\
      \bottomrule
   \end{tabular}
   \caption{Performance of the PRSC preconditioner for the Maxwell equation}
   \label{tab:maxwell-scal}
\end{table}

\begin{table}[htbp]
   \centering
   \begin{tabular}{@{} rrcccccc @{}} 
         \toprule
\multirow{2}{*}{DOF} & \multirow{2}{*}{\#Cores}  & \multicolumn{3}{c}{Error-Free} &  \multicolumn{3}{c}{With Error} \\
                                \cmidrule(lr){3-5} \cmidrule(lr){6-8}
            &     & \#Iter   & Time  & Efficiency & \#Iter       &     Time  &Efficiency \\
\hline
206,155       & 8   &   7      & 8.65   &  100\% & 8 &   8.88   &  100\% \\
436,515       & 16  &   9      & 11.01  &  78.6\% & 10 &   11.35   & 78.2\% \\
823,875       & 32  &   9      & 18.99  &  45.6\% & 11 &   19.47   & 45.6\% \\
1,610,307     & 64  &   12     & 20.48  &  42.2\% & 12 &   20.77  & 42.8\%  \\
3,416,643     & 128 &   11     & 24.14  &  35.8\% & 12 &   26.06   & 34.1\% \\
6,440,067     & 256 &   17     & 30.42  &  28.4\% & 18 &   31.92   & 27.8\% \\
12,731,523    & 512 &   21     & 33.74  & 25.6\% &  22 &   34.98   & 25.4\% \\
      \bottomrule
   \end{tabular}
   \caption{Performance of the PRSC preconditioner for the linear elasticity equation}
   \label{tab:els-scal}
\end{table}

%
%!TEX root = main.tex

\section{Concluding remarks}\label{sec:conclusion}

In this paper, we discussed a new approach to introduce local redundancy to iterative linear solvers to improve their error-resilience---We introduce redundant subspaces to the method of subspace corrections and they, in turn, can improve the resilience of the iterative procedure as well as the preconditioning step. The redundant subspace correction methods can be combined with other error detection and correction mechanisms on different levels of the system stack to improve the mean time to failure of extreme-scale computers. Exploring the intrinsic fault-tolerant features of the iterative solvers (and other numerical schemes) can open a new door to improve reliability of long-running large-scale PDE applications. We presented preliminary numerical examples to demonstrate the advantages and potentials of the proposed approach. Although our numerical tests are based on the one-level domain decomposition method, multilevel redundant subspace correction methods can be developed to improve convergence and it will be our future topic of research.
%
%\input{acknowledgements}
%%%%%%%%%%%%%%%%%%%%%%%%
\printbibliography
%%%%%%%%%%%%%%%%%%%%%%%%

\end{document}